\documentclass[12pt]{amsart}
\usepackage{amsmath,amssymb,amsfonts, amscd}
\usepackage{amsbsy}
\usepackage{amsthm}
\usepackage{amstext}
\usepackage{amsopn}
\usepackage{hyperref}
\usepackage{mathrsfs}
\usepackage{graphicx}
\usepackage{latexsym}
\usepackage[T1]{fontenc}
\usepackage{color}

\newcommand{\Hmm}[1]{\leavevmode{\marginpar{\tiny%
$\hbox to 0mm{\hspace*{-0.5mm}$\leftarrow$\hss}%
\vcenter{\vrule depth 0.1mm height 0.1mm width \the\marginparwidth}%
\hbox to 0mm{\hss$\rightarrow$\hspace*{-0.5mm}}$\\\relax\raggedright #1}}}

\newcommand{\ggrad}{{\rm{grad}}}

\newtheorem{thm}{Theorem}[section]

\newtheorem{pro}[thm]{Proposition}

\theoremstyle{definition}

\newtheorem{rem}[thm]{Remark}

\newcommand{\R}{{\mathbb R}}

\newcommand{\N}{{\mathbb N}}

\newcommand{\lm}{{\lambda}}

\newcommand{\vol 	}{{\mathrm{vol}}}

\begin{document}
\title[Eigenvalue bounds for non-compact manifolds]{ A note on eigenvalue bounds for non-compact manifolds}

\author[M. Keller]{Matthias Keller}
\address[M. Keller]{Universit\"at Potsdam, Institut f\"ur Mathematik, 14476  Potsdam, Germany} \email{matthias.keller@uni-potsdam.de}

\author[S. Liu]{Shiping Liu}
\address[S. Liu]{School of Mathematical Sciences, University of Science and Technology of China,
	Hefei 230026, Anhui Province, China} \email{
	spliu@ustc.edu.cn}

\author[N.~Peyerimhoff]{Norbert Peyerimhoff}
\address[N.~Peyerimhoff]{Department of Mathematical Sciences, Durham
	University, Science Laboratories South Road, Durham, DH1 3LE, UK}
\email{norbert.peyerimhoff@durham.ac.uk}

\begin{abstract}
  In this article we prove upper bounds for the Laplace eigenvalues
  $\lambda_k$ below the essential spectrum for strictly negatively curved
  Cartan-Hadamard manifolds. Our bound is given in terms of $k^2$ and
  specific geometric data of the manifold. This applies also to the
  particular case of non-compact manifolds whose sectional curvature
  tends to $-\infty$, where no essential spectrum is present due to a
  theorem of Donnelly/Li. The result stands in clear contrast to
  Laplacians on graphs where such a bound fails to be true in
  general.
\end{abstract}

\maketitle

\section{Introduction}
In 1979 Donnelly and Li \cite{DL} proved a criterion for discrete spectrum of the Laplacian on Riemannian manifolds in terms of decreasing sectional curvature. This complemented a result by Weyl for Schr\"odinger operators with increasing potential.

In particular, let $ M $ be a complete Riemannian manifold and $ \Delta $ be the Laplacian. We denote by $ K_{r} $ the supremum of the sectional curvatures at points \emph{outside} of $ B_{r}(x_{0}) $, the ball of radius $ r $ about some arbitrary base point $ x_{0} $, that is
\begin{equation} \label{eq:Kr} 
K_r := \sup \{ K(\sigma) \mid \text{$\sigma \subset T_pM$ two-dim. subspace, $p \in M \backslash B_{r}(x_{0})$} \}. 
\end{equation}
Then the theorem of Donnelly/Li reads as follows.

\begin{thm}[Donnelly/Li]
	Let $ M $ be a complete simply connected negatively curved Riemannian manifold. If $ K_{r}\to-\infty $ as $ r\to\infty $, then $ \Delta $ has purely discrete spectrum.
\end{thm}

In this note we give an upper bound on the eigenvalues $ \lm_{k} $ (listed with increasing order and counting multiplicities) in terms of $ k^{2} $ and specific geometric data of the manifold.  While this bound is a classical result in the case of compact manifolds, it stands in clear contrast to case of Laplacians on graphs. Indeed, for graphs any asymptotics of eigenvalues can occur, see e.g. \cite{BGK15}. 

Our result is based on so-called \emph{improved Cheeger inequalities} which were introduced in the setting of finite graphs in \cite{KLLOT13}. A dimen\-sion-free version of these improved Cheeger inequalities in the manifold setting was derived in \cite{Liu14} to prove an eigenvalue ratio result for closed weighted manifolds of \emph{non-negative Bakry-\'Emery curvature}. In this article, we discuss an application in the case of \emph{negative curvature}: we use an adaption of the improved Cheeger inequalities for general non-closed manifolds (Theorem \ref{p:bound}) to derive the following result on eigenvalues below the essential spectrum for strictly negatively curved Cartan-Hadamard manifolds:  

\begin{thm}\label{t:main}
	Let $ M $ be a complete simply connected Riemannian manifold 
	with strictly negative curvature, that is $K_{0} < 0$ (with $K_r$ defined
	in \eqref{eq:Kr}). Then, we have for all $ L^{2} $-eigenvalues 
	$ \lm_{0}< \lm_{1}\leq\ldots$ of the Laplacian below the essential
	spectrum
	$$
	\lm_{k} \leq \frac{128 \mu^{2}}{|K_{0}|(\dim(M)-1)^{2}} k^{2}, 
	\qquad k\ge 1,
	$$
	where  
	$$
	\mu=\inf_{r,s>0,x\in M}\frac{\vol(B_{r+s}(x)\setminus B_{s}(x))}{r^{2}\vol 	(B_{s}(x))}. 
	$$
\end{thm}

\begin{rem} 
Using the result of Cheng \cite{Cheng}, one can obtain a different upper bound as follows. For a ball $ B_r(x) \subset M $ with lower Ricci curvature bound larger than $(n-1)R $ with $R<0 $, and   $n=\dim (M) $, Cheng obtains for the Dirichlet eigenvalues of this ball
$$
\lm_{k}(B_r(x))\leq \frac{n^2}{4}|R|+\frac{(1+\pi^2)(1+2^{4n})}{r^2} k^{2}
$$
for odd dimensions and an estimate with somewhat better constants for the even-dimensional case and all $k \ge 0$, see \cite[Corollary~2.3]{Cheng} and \cite[Theorem~7, Chapter III]{Chavel}. (Note that Cheng proves this result for closed manifolds but his arguments work also without modification in the case of the compact manifold $B_r(x)$ with Dirichlet boundary conditions. Note also that under the assumptions of Theorem \ref{t:main}, we have $R \le K_0$.)
By domain monotonicity,  \cite[Corollary~1, Chapter I]{Chavel}, we have
for all eigenvalues $\lambda_k(M)$ of $M$ below the essential spectrum
\begin{align*}
\lm_{k}(M)\leq \lm_{k}(B_r(x)).
\end{align*}
This yields an  upper estimate with different geometric constants.
\end{rem}

We introduce the following notation. For a Riemannian manifold $ M $ let $ \vol $ be its volume measure and $ d $ the Riemannian distance. For a Borel set $ A\subseteq M $ the boundary measure $ \vol^{+}(A) $ is defined as
$$
\vol^{+}(A)=\liminf_{r\to 0}\frac{\vol(O_{r}(A))-\vol(A)}{r},
$$
where $ O_{r}(A)=\{x\in M\mid d(x,a)\leq r\mbox{ for some } a\in A\} $. If $ A $ has positive volume and finite boundary measure, we let
$$
\phi(A)=\frac{\vol^{+}(A)}{\vol(A)}
$$
and $ \phi(A)=\infty $ otherwise. The Cheeger constant of a non-compact Riemannian manifold $M$ is defined as (see \cite[p. 95]{Chavel})
$$
h=h(M)=\inf_{A\subseteq M} \phi(A).
$$

We deduce the theorem above from the following result for general manifolds which was shown in the setting of closed manifolds, \cite[Theorem~1.6]{Liu14}. The basic idea of the proof is an extension of the methods of \cite[Lemma 4, Proposition 2]{KLLOT13} developed for finite graphs to prove the so-called improved Cheeger inequalities.

\begin{thm}\label{p:bound}
Let $ M $ be a complete Riemannian manifold. Then, we have for all 
$ L^{2} $-eigenvalues $ \lm_{0}\, \leq \lm_{1} \leq\, \ldots$ of the Laplacian below the essential spectrum
$$
{h^{2}}\lm_{k}\leq 128 k^{2}\lm_{0}^{2} ,\qquad  k\ge1.
$$
\end{thm}

The proof of the theorem is based on an estimate which was proven for compact manifolds in \cite[Theorem~3.1]{Liu14}. Although the proof carries over directly we recall the proof here for the convenience of the reader.
To this end let $ f\ge0 $ be a function  on $ M $ that is supported on a set of positive measure and define let
$$  \phi(f)=\inf_{t\ge0}\phi(M_{f}(t)),  $$
where
$ M_{f}(t)= \{x\in M\mid f(x)>t\}$ is the level set of $ f $ for $ t\in\R$.
Furthermore, we denote the $ L^{p} $ norm by $ \|\cdot\|_{p} $ for $ p\in[1,\infty] $.

\begin{pro}[Non-compact version of Theorem~3.1 \cite{Liu14}]	
Let $ M $ be a complete Riemannian manifold with $ L^{2} $-eigenvalues 
$$ \lm_{0}\, \leq\,  \lm_{1}\leq\ldots $$ 
of the Laplacian below the essential spectrum and let $ f\ge 0 $ be a bounded Lipshitz function in $ L^{2}(M) $. Then,
$$
\phi(f)\leq 8\sqrt{2}\frac{k }{\sqrt{\lm_{k}}}\frac{\||\nabla f|\|^{2}_{2}}{\| f\|^{2}_{2}},\qquad  k\ge1.
$$	
\end{pro}

\begin{proof}
Here we sketch the core arguments of the proof. For more details we refer 
the reader to  \cite{Liu14}. We assume    $|\nabla f|\in L^{2}(M) $ since otherwise the asserted inequality is trivial.

For a finite set $ \theta\subset\R $, let $ \psi_{\theta}:\R\to \R $ 
be defined by
$$
\psi_{\theta}(s)=\arg\min_{t\in \theta}|s-t|,
$$
$\eta_\theta: \R\to \R$ be defined by
$$
\eta_{\theta}(s)=|s-\psi_{\theta}(s)|,
$$
and $ \eta_{\theta,f}:M\to[0,\infty) $
$$  
\eta_{\theta,f}= \eta_{\theta} \circ f = |f-\psi_{\theta}\circ f|
$$
be the difference of $ f $ and its approximation $ \psi_{\theta}\circ f $. 
Note that we have $0 \le \eta_{\theta,f} \le f$.
   
Now, fix $ k\in\N $ for the rest of the proof and let $ t_{0}=0 $. Assume   
$ t_{0}<t_{1}<\ldots<t_{j-1}$ are given. If there is $ t\ge t_{j-1} $ such that
\begin{equation} \label{eq:relevant}
\| \eta_{\{t_{j-1},t\},f}1_{f^{{-1}}  ((t_{j-1},t])} \|^{2}_{2} =
\frac{1}{k\lm_{k}}\||\nabla f|\|^{2}_{2} =:C_{0},
\end{equation}
then let $ t_{j} $ be the smallest such $ t\ge t_{j-1} $. Otherwise, let 
$ t_{j}=\|f\|_{\infty}$. 
Observe that  
$$
f_{j}=  \eta_{\{t_{j-1},t_{j}\},f}1_{f^{{-1}}  ((t_{j-1},t_{j}])},\qquad j\ge 1,
$$
are positive disjointly supported Lipshitz functions which are trivial whenever 
$ t_{j}=\|f\|_{\infty}$. Moreover, $ f_{j}\in L^{2} $ since $0 \le f_j \le f$ and $f \in L^2(M)$. Furthermore, by the reverse triangle inequality we have 
$ | f_{j} (x)-  f_{j}(y) |\leq |f(x) -f(y)|$, $ x,y\in M $. Therefore, as the supports 
of the $ f_{j} $ are disjoint, we obtain  
$$
\sum_{j=1}^{\infty}|\nabla f_{j}|^{2}\leq |\nabla f|^{2}   
$$
and therefore, $ |\nabla f_{j}|\in L^{2}(M) $ whenever $ |\nabla f|\in  L^{2}(M) $.   By completeness of the Riemannian manifold, the Laplacian is essentially selfadjoint. Thus, the $ f_{j} $'s are included in the form domain of the Laplacian since $ f_{j},|\nabla f_{j}|\in L^{2}(M) $.    
We show the following claim.
  
\smallskip
  
\emph{Claim:} $ t_{2k}=\|f\|_{\infty} $.  

\smallskip

In the case $ t_{2k}<\|f\|_{\infty} $, we infer by the arguments above and by the fact that in this case $ \| f_{j}\|^{2}_{2} = C_{0}$  for all $ j=1,\ldots,2k $
$$
\sum_{j=1}^{2k} \frac{\||\nabla f_{j}|\|^{2}_{2} }{\| f_{j}\|^{2}_{2} }\leq \frac{1}{C_{0}} \||\nabla f	|\|^{2}_{2}=k\lm_{k}.
$$
By the assumption $ t_{2k}<\|f\|_{\infty} $, the functions $ f_{j} $ are non-zero and therefore non-constant. Thus, there exist at least $ k+1 $ of the $ f_{j} $'s such that 
$$
\frac{\||\nabla f_{j}|\|^{2}_{2}}{\| f_{j}\|^{2}_{2}}\leq \lm_{k}.
$$
Hence, the inequality  above for $ k+1 $ orthogonal functions stands in contradiction the Min-Max-Principle and the claim is proven.

\smallskip
  
So let $\theta=\{ 0=t_{0}<t_{1}\leq \ldots\leq t_{2k}=\|f\|_{\infty}\} $.
By \eqref{eq:relevant} and what we have shown above, we obtain
\begin{equation}\label{ie1}
\|f-\psi_{\theta}\circ f\|^{2}_{2} =\sum_{j=1}^{2k}\| \eta_{\{t_{j},t\},f}1_{f^{{-1}}  ((t_{j},t])}\|^{2}_{2}\leq \frac{2}{\lm_{k}}\||\nabla f|\|^{2}_{2} .
\end{equation}
In order to estimate  the $ L^{2} $ norm of $ f-\psi_{\theta}\circ f $ from below, we observe  that the function $ h:M\to\R $
$$
h(x)=\int_{0}^{f(x)}\eta_{\theta}(t)dt
$$
has the same level sets as $ f $ and therefore,
$$
\phi(f)=\phi(h) \leq \frac{\||\nabla h|\|_{1}}{\|h\|_{1}},
$$
where the last inequality follows from the area formula and the co-area inequality \cite[Lemma~3.2]{BH97} (for more details see \cite[Lemma 2.4]{Liu14}). Firstly, we find by the fundamental theorem of calculus and the chain rule and secondly by the Cauchy-Schwarz inequality and 
$ \eta_{\theta,f}=f-\psi_{\theta}\circ f $ that
$$
\||\nabla h|\|_{1}=\||\nabla f|(\eta_{\theta}\circ f)\|_{1}\leq  
\||\nabla f|\|_{2}\|f-\psi_{\theta}\circ f|\|_{2}.
$$
Thirdly, is it elementary to estimate 
$$
h\ge \frac{1}{8k}f^{2}
$$
by choosing $ t_{j}\leq f(x)\leq t_{j+1} $ for $ x \in M$ and estimating
\begin{align*}
h(x)&\ge\frac{1}{4}\left( \sum_{i=0}^{j-1}{(t_{i+1}-t_{i})^{2}} +
{(f(x)-t_{j})^{2}}\right)\\
&\ge\frac{1}{8k}\left( \sum_{i=0}^{j-1}{(t_{i+1}-t_{i})} +{(f(x)-t_{j})}\right)^{2}=\frac{1}{8k}f(x)^{2}.
\end{align*}
These considerations together with \eqref{ie1}  yield
$$
\phi(f)\leq \frac{\||\nabla h|\|_{1}}{\|h\|_{1}}\leq 8k 
\frac{\||\nabla f|\|_{2}\|f-\psi_{\theta}\circ f|\|_{2}}{\|f\|_{2}^{2}}
\leq 8\sqrt{2}\frac{k }{\sqrt{\lm_{k}}}\frac{\||\nabla f|\|_{2}^{2}}{\|f\|_{2}^{2}},
$$
which finishes the proof.  
\end{proof}

With the help of this proposition we are now in the position to prove Theorem~\ref{p:bound}.
 
\begin{proof}[Proof of Theorem~\ref{p:bound}]
We observe that for any $ n $ we have
$$
\phi(f)\leq \phi(f\wedge n),
$$
where $ f\wedge n =\min\{f,n\}$. Moreover, by the proposition above we have
$$
\phi(f\wedge n)\leq 8\sqrt{2}\frac{k }{\sqrt{\lm_{k}}}\frac{\int_{M}|\nabla f\wedge n|^{2}d\vol}{\int_{M}|f\wedge n|^{2}d\vol}.
$$
Since $ \phi(f)\leq \phi(f\wedge n) $, $ |\nabla (f\wedge n)|\leq |\nabla f| $ and $ \int_{M} |f\wedge n|^{2}d\vol  \to\int_{M}|f|^{2}d\vol=1$, $ n\to\infty $, we conclude
$$
\phi(f)\leq 8\sqrt{2}\frac{k }{\sqrt{\lm_{k}}}\frac{\int_{M}|\nabla f|^{2}d\vol}{\int_{M}|f|^{2}d\vol}	.
$$
We choose $f$ to be an  eigenfunction to $ \lm_{0} $. Then, $ f $ is a Lipshitz function in $ L^2(M) $ with a definite sign which can be chosen to be positive. Then, by the definition of the Cheeger constant and the proposition above, we have
$$
h\sqrt{\lm_{k}}\leq  8\sqrt{2}k \lm_{0},
$$
which finishes the proof.
\end{proof}

\begin{proof}[Proof of Theorem~\ref{t:main}]
Let us first derive $h^2 \ge (\dim(M)-1)^{2}|K_{0}|$: In the
definition of the Cheeger constant, we can restrict ourselves to
sets $A$ with smooth boundary. Let $A \subset M$ be such a set,
$\widehat x \in M$ be a point with positive distance to $A$, and
$d_{\widehat x}: M \to [0,\infty)$ be the distance function to
$\widehat x$. Then $d_{\widehat x}$ is a smooth function on $A$
(since the exponential map $\exp_{\widehat x}: T_{\widehat x}M \to M$
is a diffeomorphism). By the Laplacian Comparison Theorem (see, e.g.,
\cite[(3)]{Kura}), we have
$$
\Delta_M\, d_{\widehat x}(x) \ge (\dim(M)-1) \sqrt{-K_{0}}
\coth(\sqrt{-K_{0}} d_{\widehat x}(x)).
$$
This implies that $\Delta_M\, d_{\widehat x}(x) \ge (\dim(M)-1) \sqrt{|K_{0}|}$
for all $x \in A$ and, therefore, on the one hand,
$$ \int_A \Delta_M\, d_{\widehat x}\, d\vol \ge (\dim(M)-1)\, \sqrt{|K_{0}|}\, \vol(A), $$
and, on the other hand, using the Gau{\ss} Divergence Theorem,
$$ \int_A \Delta_M\, d_{\widehat x}\, d\vol = \int_{\partial A} \langle \ggrad\,
d_{\widehat x}, \nu \rangle\, d\vol_{\partial A} \le \vol^+(A), $$
where $\nu$ is the outward unit normal vector of $\partial A$. Combining
both inequalities leads to the proof of the above estimate of the Cheeger
constant $h$.

Furthermore, let $ \eta=(1-d(\cdot,B_{s}(x))/r)_{+} $. Then,
\begin{align*}
\lm_{0}&\leq \frac{\int_{M}|\nabla \eta|^{2}d\vol}{\int_{M}|\eta|^{2}d\vol}\\
&=\frac{\vol (B_{r+s}(x)\setminus B_{s}(x))} {r^{2}(\vol(B_{s}(x))+
\int_{ (B_{r+s}(x))\setminus B_{s}(x)}(r-d(y,B_{s}(x))^{2}d\vol(y))}\\
&\leq \frac{\vol (B_{r+s}(x)\setminus B_{s}(x))} {r^{2}\vol(B_{s}(x))}.
\end{align*}
Hence, combining this with Proposition~\ref{p:bound} we conclude the statement of the theorem.	
\end{proof}

\textbf{Acknowledgement}.  The authors enjoyed the hospitality of
TSIMF where this work was realized. MK acknowledges the financial
support of the German Science Foundation (DFG). The authors are also
grateful to Gerhard Knieper and the anonymous referees for helpful comments.

\newcommand{\etalchar}[1]{$^{#1}$}
\def\cprime{$'$} \def\cprime{$'$}

\end{document}